\documentclass[a4paper,reqno]{amsart}

\usepackage{geometry}
\title{Infinitely Many Tangent Functors on Diffeological Spaces}
\author{Masaki Taho}                                                                               

\makeatletter
\@namedef{subjclassname@2020}{\textup{2020} Mathematics Subject Classification}
\makeatother
\subjclass[2020]{Primary~57P05; Secondary~58A05}

\keywords{diffeology, tangent space}

\address{GRADUATE SCHOOL OF MATHEMATICAL SCIENCES, THE
UNIVERSITY OF TOKYO, 3-8-1 KOMABA, MEGURO-KU, TOKYO, 153-8914,
JAPAN}
\email{taho@ms.u-tokyo.ac.jp}

\usepackage{amsmath} 
\usepackage{amssymb} 
\usepackage{amsfonts} 
\usepackage{mathrsfs} 
\usepackage{amsthm} 
\usepackage{bm} 
\usepackage[new]{old-arrows} 
\usepackage[dvipdfmx]{graphicx} 
\usepackage{wrapfig} 
\usepackage{color} 
\usepackage[labelsep=colon]{caption} 
\usepackage[labelformat=empty,subrefformat=parens]{subcaption} 

\usepackage{spalign}

\usepackage{multicol} 
\usepackage{units} 
\usepackage[all]{xy} 
\usepackage{tikz-cd} 
\usepackage[backref]{hyperref} 
\hypersetup{
    colorlinks=true,
    linkcolor= [rgb]{0.3,0.55,0.25},
    citecolor=[rgb]{0.164,0.39,0.77},
    urlcolor=[rgb]{0,0,1}
}

\usepackage[nameinlink,capitalize,noabbrev]{cleveref}

\crefname{thm}{theorem}{theorems}
\Crefname{thm}{Theorem}{Theorems}
\crefname{prop}{proposition}{propositions}
\Crefname{prop}{Proposition}{Propositions}
\crefname{lemma}{lemma}{lemmas}
\Crefname{lemma}{Lemma}{Lemmas}
\crefname{cor}{corollary}{corollaries}
\Crefname{cor}{Corollary}{Corollaries}
\crefname{defi}{definition}{definitions}
\Crefname{defi}{Definition}{Definitions}
\crefname{ex}{example}{examples}
\Crefname{ex}{Example}{Examples}
\crefname{remark}{remark}{remarks}
\Crefname{remark}{Remark}{Remarks}
\crefname{q}{question}{questions}
\Crefname{q}{Question}{Questions}
\crefname{conj}{conjecture}{conjectures}
\Crefname{conj}{Conjecture}{Conjectures}
\crefname{fact}{fact}{facts}
\Crefname{fact}{Fact}{Facts}

\crefname{equation}{equation}{equations}
\Crefname{equation}{Equation}{Equations}
\crefname{figure}{figure}{figures}
\Crefname{figure}{Figure}{Figures}
\crefname{table}{table}{tables}
\Crefname{table}{Table}{Tables}
\crefname{section}{section}{sections}
\Crefname{section}{Section}{Sections}

\usepackage{thmtools}
\usepackage{thm-restate}  

\newtheorem{thm}{Theorem}[section]
\newtheorem{prop}[thm]{Proposition}

\newtheorem*{mthm*}{Main Theorem}
\newtheorem*{thm*}{Theorem}
\newtheorem*{prop*}{Proposition}

\theoremstyle{definition}
\newtheorem{defi}[thm]{Definition}
\newtheorem{ex}[thm]{Example}
\newtheorem{remark}[thm]{Remark}

\newcommand{\dflg}{{\mathrm {Dflg}}}                 
\newcommand{\dflgloc}{{\mathrm {Dflg}}_*^{loc}}      
\newcommand{\enclzeroloc}{{\mathrm {Eucl}_0^{loc}}}  
\newcommand{\vect}{{\mathrm {Vect}}}                 

\newcommand{\inclzeroloc}{i_0^{loc}}                 

\newcommand{\dom}{\operatorname{dom}}            
\newcommand{\plotdom}[1]{U_{#1}}                      

\newcommand{\Lan}{\operatorname{Lan}}                   
\newcommand{\Ran}{\operatorname{Ran}}                   

\newcommand{\id}{\operatorname{id}}

\newcommand{\nattrans}{\operatorname{\eta}^R}

\newcommand{\germ}{G}                                
\newcommand{\plots}{{\mathrm {Plots}}}                           

\newcommand{\internalfunctor}{T}                     
\newcommand{\internal}[2]{T_{#1} (#2)}               
\newcommand{\rightfunctor}{\hat{T}^R}
\newcommand{\righttangent}[2]{\hat{T}^R_{#1} (#2)}

\newcommand{\vincenttangent}[2]{\hat{T}^{\mathrm{Vin}}_{#1} (#2)}

\newcommand{\generalinternal}[3]{T^{#1}_{#2} (#3)}

\newcommand{\generalright}[3]{\hat{T}^{#1}_{#2} (#3)}
\newcommand{\generalinternalfunctor}[1]{T^{#1}}

\newcommand{\generalrightfunctor}[1]{\hat{T}^{#1}}

\newcommand{\Hom}{\mathrm{Hom}}

\begin{document}

\begin{abstract}
We study tangent spaces in the setting of diffeological spaces. Several distinct tangent functors have been introduced, each of which extends the classical tangent functor from smooth manifolds. In this paper, we construct infinitely many non-isomorphic tangent functors on diffeological spaces. We compare our constructions with existing models, including the internal and external tangent spaces. Our results show that the choice of tangent functor is far from unique outside smooth manifolds.
\end{abstract}

\maketitle

\section{Introduction}

The tangent functor is a fundamental construction in differential geometry, 
providing the foundation for vector fields, flows, and differential forms.

In the broader setting of diffeological spaces introduced by Souriau \cite{Souriau},
it is natural to define a tangent functor that extends the classical tangent functor.
In this paper, a \emph{tangent functor} means a functor 
$T\colon \dflgloc\to\vect$ that agrees with the classical tangent functor on smooth manifolds. 
Here $\dflgloc$ is the category of based diffeological spaces with morphisms given by germs of basepoint-preserving smooth maps, and 
$\vect$ is the category of real vector spaces and linear maps. 
See \Cref{section:tangent-functors} for the precise axioms.
Even with this setup, such extensions are far from unique, and several non-isomorphic constructions have already been proposed.
An early construction, due to Iglesias-Zemmour \cite{IZ}, is based on differential 1-forms. 
Vincent \cite{vincent} later proposed another construction, which is closely related to the internal and external tangent spaces introduced by Christensen-Wu \cite{CW} and can be seen as an intermediate construction between them.
Further variants of the external construction were given in \cite{taho}.

\medskip

While previous work has produced only finitely many non-isomorphic tangent functors, 
we show that the variety of tangent functors is far richer than previously known.
Specifically, for each based diffeological space $(Y,y)$ with a nontrivial tangent vector at $y$, 
we construct two tangent functors, both of which agree with the classical tangent functor on smooth manifolds:
\[
  \generalinternalfunctor{(Y,y)},\ \generalrightfunctor{(Y,y)}\colon\dflgloc\longrightarrow\vect.
\]

To illustrate the diversity of such functors, 
we take different test spaces $(Y,y)$ for the two constructions. 
For the first construction, we take $Y$ to be the irrational torus
\[
T_\alpha=\mathbb{R}/(\mathbb{Z}+\alpha\mathbb{Z}), \quad \text{where }\alpha\in\mathbb{R}\setminus\mathbb{Q}.
\]
This choice yields uncountably many pairwise non-isomorphic tangent functors.
For the second construction, setting $(Y,y)=(H_n,0)$ with $H_n=\mathbb{R}^n/O(n)$
also gives infinitely many non-isomorphic tangent functors.

\begin{restatable*}{thm}{maintheoremone}\label{thm:mainone}
  Let $\alpha,\beta\in\mathbb{R}\setminus\mathbb{Q}$.
  Then for any $x\in T_\alpha$ and $y\in T_\beta$, we have
  \[
    \generalinternal{(T_\beta, y)}{x}{T_\alpha} \cong
    \begin{cases}
      \mathbb{R}, & \text{if there exists a nonconstant smooth map } T_\alpha\to T_\beta,\\[2mm]
      0, & \text{otherwise.}
    \end{cases}
  \]
  Equivalently, by \Cref{ex:irrational-torus}, we have
  \[
    \generalinternal{(T_\beta, y)}{x}{T_\alpha}\ \cong\
    \begin{cases}
      \mathbb{R}, & \text{if }\ \alpha=\dfrac{a+b\,\beta}{c+d\,\beta}\ \text{for some }a,b,c,d\in\mathbb{Z},\\[3mm]
      0, & \text{otherwise.}
    \end{cases}
  \]
\end{restatable*}

\begin{restatable*}{thm}{maintheoremtwo}\label{thm:maintwo}
  Let $n,m\in\mathbb{Z}_{\geq 1}$.
  Then we have
  \[
    \generalright{(H_m,0)}{0}{H_n} \cong
    \begin{cases}
      \mathbb{R}, & \text{if } m\leq n,\\[2mm]
      0, & \text{if } m>n.
    \end{cases}
  \]
\end{restatable*}

These results show that, unlike in the manifold case where the tangent functor is essentially unique, 
the notion of a tangent functor on diffeological spaces requires more careful treatment.


\section{Diffeological spaces and tangent spaces}\label{section:diffeology}

We briefly recall the notions of diffeological spaces and tangent spaces on diffeological spaces.
For a detailed exposition, see \cite{IZ}, \cite{CW}, \cite{vincent}, \cite{taho}.

\subsection{Diffeological spaces}
\begin{defi}[\cite{IZ}, 1.5] \label{definition diffeology}
    Let $X$ be a set. A \emph{parametrization} of $X$ is a map $U\to X$ where $U$ is an open set of $\mathbb{R}^n$ for some $n\in\mathbb{Z}_{\geq 0}$.
    A \emph{diffeology} on $X$ is a set of parametrizations $\mathscr{D}_X$ (we call an element of $\mathscr{D}_X$ a plot) such that the following three axioms are satisfied:
    \begin{description}
      \item[Covering] Every constant parametrization $U\to X$ is a plot.
      \item[Locality] Let $p\colon U\to X$ be a parametrization. If there is an open covering $\{U_{\alpha}\}$ of $U$ such that ${p|}_{U_{\alpha}}\in \mathscr{D}_X$ for all $\alpha$, then $p$ itself is a plot.  
      \item[Smooth compatibility] For every plot $p\colon U\to X$, every open set $V$ in $\mathbb{R}^m$ and $f\colon V\to U$ that is smooth as a map between Euclidean spaces, $p\circ f\colon V\to X$ is also a plot.
    \end{description}
    A diffeological space is a set equipped with a diffeology.
    We usually write $X$ for the diffeological space $(X,\mathscr{D}_X)$.
\end{defi}

\begin{defi}[\cite{IZ}, 1.14]
A map $f\colon X\to Y$ between diffeological spaces is \emph{smooth} if for every plot $p\colon U\to X$, the composition $f\circ p$ is a plot of $Y$.
\end{defi}

Diffeological spaces and smooth maps form a category $\dflg$ that contains smooth manifolds as a full subcategory.

\begin{ex}
If $M$ is a smooth manifold, the set of all smooth maps $U\to M$ from Euclidean open sets $U$ forms a diffeology, called the \emph{standard diffeology} on $M$.
\end{ex}

\begin{defi}[\cite{IZ}, 2.8]\label{def:D-topology}
  Let $X$ be a diffeological space. The \emph{$D$-topology} on $X$
  is the finest topology that makes all plots $p\colon U\to X$ continuous.
\end{defi}  

\begin{defi}[\cite{IZ}, 1.50]\label{def:quotient}
Let $X$ be a diffeological space and $\sim$ an equivalence relation on $X$.  
The \emph{quotient diffeology} on $Y=X/\!\sim$ is the smallest diffeology (with respect to inclusion of plots) for which the quotient map $\pi\colon X\to Y$ is smooth.  
Equivalently, a parametrization $p\colon U\to Y$ is a plot if and only if each point of $U$ has a neighborhood $U'$ such that $p|_{U'}=\pi\circ q$ for some plot $q\colon U'\to X$.
\end{defi}

\begin{ex}[Irrational tori, {\cite[Exercise~4]{IZ}}]\label{ex:irrational-torus}
  Let $\alpha\in\mathbb{R}\setminus\mathbb{Q}$.  
  The \emph{irrational torus} of slope $\alpha$ is the quotient diffeological space
  \[
    T_\alpha=\mathbb{R}/(\mathbb{Z}+\alpha\mathbb{Z})
  \]
  equipped with the quotient diffeology induced from $\mathbb{R}$.  
  The following facts are established in {\cite[Exercise~4]{IZ}}:
  \begin{itemize}
    \item
    Every smooth map $f\colon T_\alpha\to\mathbb{R}$ is constant, i.e.,
    $C^\infty(T_\alpha,\mathbb{R})\cong\mathbb{R}$. 
    Indeed, for the quotient projection $\pi_\alpha\colon \mathbb{R}\to T_\alpha$,
    the composition $g=f\circ\pi_\alpha$ satisfies $g(t+1)=g(t)$ and $g(t+\alpha)=g(t)$,
    which forces $g$ to be constant since $\mathbb{Z}+\alpha\mathbb{Z}$ is dense in $\mathbb{R}$.
    \item
    For $\alpha, \beta\in \mathbb{R}\setminus \mathbb{Q}$, every smooth map $f\colon T_\alpha\to T_\beta$ arises from an affine map 
    $F:\mathbb{R}\to\mathbb{R}$ such that
    $\pi_\beta\circ F = f\circ \pi_\alpha$. Moreover, $f$ is constant if and only if $F$ is constant.
    There exists a nonconstant smooth map $T_\alpha\to T_\beta$ if and only if
    \[
      \alpha=\frac{a+b\,\beta}{c+d\,\beta}
    \]
    for some $a,b,c,d\in\mathbb{Z}$. 
    \item
    Two irrational tori $T_\alpha$ and $T_\beta$ are diffeomorphic if and only if
    there exist $a,b,c,d\in\mathbb{Z}$ with
    \[
      \alpha=\frac{a+b\,\beta}{c+d\,\beta},\quad \text{ and }\quad ad-bc=\pm1.
    \]
  \end{itemize}
\end{ex}

\medskip

\subsection{The internal tangent spaces}

We now recall the internal tangent spaces introduced by Christensen-Wu \cite{CW}. 
Let $\dflgloc$ denote the category whose objects are based diffeological spaces $(X,x)$
and whose morphisms are germs (with respect to the $D$-topology) of basepoint-preserving smooth maps.
Let $\enclzeroloc$ be the full subcategory of $\dflgloc$
whose objects are pairs $(U,0)$, where $U$ is an open subset of some $\mathbb{R}^n$.
Note that $\enclzeroloc$ can equivalently be regarded as
the full subcategory of $\dflgloc$ consisting of based smooth manifolds. 
For a based diffeological space $(X,x)$, we define the category $\plots(X,x)$ by
\[
  \plots(X,x)=\enclzeroloc\!\downarrow\!(X,x),
\]
whose objects are germs of plots $p\colon (\plotdom{p},0)\to(X,x)$ centered at $x$.
Let $T^{\enclzeroloc}\colon\enclzeroloc\to\vect$ denote the classical tangent functor at the origin. See \cite[3.1]{taho} for detailed definitions of these categories and functors; for categorical background, see also Mac Lane \cite{maclane}.

\begin{defi}[Internal tangent space, {\cite[Def.~3.1]{CW}}]\label{def:internal}
The \emph{internal tangent functor} $\internalfunctor\colon\dflgloc\to\vect$ is the left Kan extension of $T^{\enclzeroloc}\colon\enclzeroloc\to\vect$ along the inclusion functor $\inclzeroloc\colon\enclzeroloc\to\dflgloc$, that is,
\[
  \internalfunctor(X,x) = \Lan_{\inclzeroloc} T^{\enclzeroloc}(X,x).
\]
Concretely, the pointwise formula of the left Kan extension yields
\[
  \internalfunctor(X,x) = \bigoplus_{p\in\plots(X,x)} T_0 (\plotdom{p})\;\Big/\;R,
\]
where
\[
  R=\langle v-w\mid \text{$v\in T_0 (\plotdom{p})$, $w\in T_0 (\plotdom{q})$, $T(f)v=w$, with $p=q\circ f$.}\rangle_{\mathrm{vect}}.
\]
Here $\langle\cdot\rangle_{\mathrm{vect}}$ denotes the vector subspace generated by the indicated elements.
We denote the vector space $\internalfunctor(X,x)$ by $\internal{x}{X}$ and call it the \emph{internal tangent space of $X$ at $x$}.
\end{defi}

For a plot $p\colon \plotdom{p}\to X$ centered at $x$ and a vector $v\in T_0(\plotdom{p})$, 
we write $[p,v]\in \internal{x}{X}$ for the class of $v\in \bigoplus_{p\in\plots(X,x)} T_0 (\plotdom{p})$ in the above quotient.

\begin{prop}[{\cite[Prop.~3.3]{CW}}]\label{prop:curves}
    The internal tangent space $\internal{x}{X}$ is generated by elements of the form $[c,\tfrac{\partial}{\partial t}]$, 
    where $c\colon (\mathbb{R},0)\to (X,x)$ is a smooth curve with $c(0)=x$ and $\tfrac{\partial}{\partial t}\in T_0(\mathbb{R})$ is the unit tangent vector at $0$.
\end{prop}

\subsection{The right tangent spaces}\label{subsec:global-right}

We next introduce the \emph{right tangent space}, 
a dual construction to the internal one in \Cref{def:internal}. 
This space can be described as the value of the right Kan extension of the classical tangent functor $T^{\enclzeroloc}\colon\enclzeroloc\to\vect$.

\begin{defi}[Right tangent space, {\cite[Def.~4.2]{taho}}]\label{def:global-right}
  Let $(X,x)$ be a based diffeological space.
  Let $\germ(X,x)$ denote the $\mathbb{R}$-algebra of germs of smooth functions at $x$.
  A map $D\colon \germ(X, x)\to\mathbb{R}$ is called a \emph{right tangent vector}
  at $x$ if it satisfies the following conditions:
  \begin{description}
    \item[Linearity]
      $D$ is $\mathbb{R}$-linear.
    \item[Leibniz rule]
      For all $g,h\in \germ(X,x)$, we have $D(gh)=g(x)D(h)+h(x)D(g)$.
  \end{description}
  The set of all right tangent vectors forms a vector space,
  denoted by $\righttangent{x}{X}$ and called the \emph{right tangent space} of $X$ at $x$.  
\end{defi}

\begin{remark}
    The right tangent space is an analog of the external tangent space \cite{CW},
    though its definition differs slightly.
    For a detailed comparison, we refer the reader to \cite{taho}.
\end{remark}  

The above definition coincides with the right Kan extension of the classical tangent functor:
\begin{prop}[{\cite[Prop.~4.15]{taho}}]\label{prop:global-right-kan}
    For any based diffeological space $(X,x)$, there is a natural isomorphism of vector spaces
    \[
      \Ran_{\inclzeroloc} T^{\enclzeroloc} (X,x) \cong \righttangent{x}{X}.
    \]
\end{prop}

\subsection{Vincent-type tangent spaces}

By the universality of the Kan extension, there is a canonical natural transformation between functors from $\dflgloc$ to $\vect$:
\[
  \nattrans=\{\nattrans_{(X,x)}\}\colon \internalfunctor \Longrightarrow \rightfunctor,
  \qquad
  \nattrans_{(X,x)}([p,v])(f)=d(\tilde f\circ p)_0(v),
\]
where $\nattrans_{(X,x)}([p,v])$ acts on each germ $f\in\germ(X,x)$
through any representative $\tilde f$ of $f$ near $x$.
We define another tangent functor as the coimage of this natural transformation.

\begin{defi}[Vincent-type tangent space]\label{def:vincent}
  For a based diffeological space $(X,x)$, the \emph{Vincent-type tangent space} is defined by
  \begin{align*}
    \vincenttangent{x}{X}
    &= \mathrm{Coim}\bigl(\nattrans_{(X,x)}\bigr)
       = \bigoplus_{p\in\plots(X,x)} T_0(\plotdom{p}) \;\Big/\; R_{\mathrm{Vin}}\\
    &\cong \mathrm{Im}\bigl(\nattrans_{(X,x)}\bigr)
       = \bigl\langle\,  D_p \;\big|\;
           p\colon (\mathbb{R},0)\to (X,x)\ \text{smooth} \,\bigr\rangle _{\mathrm{vect}}
       \subset \righttangent{x}{X}.
  \end{align*}
  Here the relation $R_{\mathrm{Vin}}$ is given by
  \[
    R_{\mathrm{Vin}}
      = \left\langle v-w \;\middle|\;
         v\in T_0(\plotdom{p}),\; w\in T_0(\plotdom{q}),\;
         d(f\circ p)(v) = d(f\circ q)(w)
         \ \text{for all } f\in\germ(X, x)
        \right\rangle_{\mathrm{vect}},
  \]
  and each $D_p\in \righttangent{x}{X}$ is the derivation defined by
  \[
    D_p(f)=\frac{d(f\circ p)}{dt}(0) \quad\text{for } f\in\germ(X, x).
  \]
\end{defi}

\begin{remark}[Relation to Vincent's definition]
If $\germ(X,x)$ in the definition of $R_{\mathrm{Vin}}$ is replaced by $C^\infty(X,\mathbb{R})$, the resulting quotient coincides with Vincent's tangent space~\cite[Def.~3.3.1]{vincent}.
\end{remark}  

\begin{remark}[Separation by germs]
  Fix $f\in\germ(X, x)$ and define a linear functional
  \[
    L_f\colon \bigoplus_{p\in\plots(X,x)} T_0(\plotdom p)\longrightarrow \mathbb{R},
    \qquad
    L_f|_{T_0(\plotdom p)}(v)= d(f\circ p)(v).
  \]
  By the definition of $R_{\mathrm{Vin}}$, every generator $v-w$ of $R_{\mathrm{Vin}}$
  satisfies $d(f\circ p)(v)=d(f\circ q)(w)$ for this fixed $f$,
  hence $L_f(v-w)=0$. Thus $L_f$ descends to a well-defined linear map
  \[
    \overline{L_f}\colon \vincenttangent{x}{X}
    =\Big(\bigoplus_{p} T_0(\plotdom p)\Big)\Big/ R_{\mathrm{Vin}}
    \longrightarrow \mathbb{R},
    \qquad
    \overline{L_f}([p,v])=d(f\circ p)(v).
  \]
  Consequently, if there exists $f\in \germ(X, x)$ such that
  $d(f\circ p)(v)\neq d(f\circ q)(w)$, then we have
  $\overline{L_f}([p,v])\neq \overline{L_f}([q,w])$,
  and hence we obtain $[p,v]\neq [q,w]$ in $\vincenttangent{x}{X}$.
  Equivalently, $[p,v]=[q,w]$ implies
  $d(f\circ p)(v)=d(f\circ q)(w)$ for all $f\in\germ(X, x)$.
\end{remark}

    \subsection{Examples}
    Before introducing our new construction, we recall how these tangent spaces behave on some basic diffeological spaces.
    \begin{ex}[Manifolds, {\cite[Example~3.16]{CW}, \cite[Example~26]{vincent}}]\label{ex:manifolds}
    Let $M$ be a smooth manifold and $x\in M$ with the standard diffeology. Then we have
    \[
    \internal{x}{M} \cong\vincenttangent{x}{M}\cong \righttangent{x}{M}.
    \]
    This follows from the fact that the inclusion functor 
    $i=\inclzeroloc\colon\enclzeroloc\hookrightarrow\dflgloc$ 
    is fully faithful.
    \end{ex}
    
    \begin{ex}[Irrational torus, {\cite[Example~3.23]{CW}}]\label{ex:int-vincent-torus}
        Let $T_\alpha$ be the irrational torus of slope $\alpha$, as defined in \Cref{ex:irrational-torus}, and let $x\in T_\alpha$. Then we have
        \[
          \internal{x}{T_\alpha} \cong \mathbb{R},
          \qquad
          \vincenttangent{x}{T_\alpha} = 0,
          \qquad
          \righttangent{x}{T_\alpha} = 0.
        \]
    \end{ex}
    
    \begin{proof}
        For the internal tangent space, {\cite[Example~3.23]{CW}} proves the isomorphism 
        and shows that the class $[\pi_\alpha,\partial/\partial t]$ serves as a generator, 
        where $\pi_\alpha\colon \mathbb{R}\to T_\alpha$ is the quotient map. 
        
        For the right tangent space, it is also shown in {\cite[Example~3.23]{CW}}
        that it vanishes.  
        Since $\vincenttangent{x}{T_\alpha}$ is a subspace of $\righttangent{x}{T_\alpha}$,
        it follows immediately that $\vincenttangent{x}{T_\alpha}=0$.
    \end{proof}

    \begin{ex}[Orbit space of the orthogonal group, {\cite[Example~3.24]{CW}}]\label{ex:orbit-space}
      Let $H_n = \mathbb{R}^n / O(n)$ be the orbit space equipped with the quotient diffeology, 
      and let $0\in H_n$ be the image of the origin. Then we have
      \[
        \internal{0}{H_n} = 0,
        \qquad
        \vincenttangent{0}{H_n} = 0,
        \qquad
        \righttangent{0}{H_n} \cong \mathbb{R}.
      \]
      \end{ex}
      
      \begin{proof}
        It is shown in {\cite[Example~3.24]{CW}} that the internal tangent space of $H_n$ vanishes at $0$,
        while the external tangent space (which coincides with the right tangent space in this case)
        is one-dimensional.
        Moreover, {\cite[Example~3.24]{CW}} identifies a generator $D_n$ of $\righttangent{0}{H_n}$ by
        \[
          D_n(f)
          \;=\;
          \frac{1}{2}\,
          \frac{\partial^2 \tilde f}{\partial x_1^2}(0),
        \]
        where $\tilde f\colon \mathbb{R}^n \to \mathbb{R}$ is any $O(n)$-invariant smooth function satisfying
        $f([x]) = \tilde f(x)$.
        In particular, for $q_n\colon H_n\to \mathbb{R};[x]\mapsto\|x\|^2$, we have $D_n(q_n) = 1$.
        Since the Vincent-type tangent space is defined as the image of the natural transformation
        $\internalfunctor \Rightarrow \rightfunctor$,
        it follows that $\vincenttangent{0}{H_n} = 0$.
      \end{proof}


\section{Infinitely many variations of tangent functors} \label{section:tangent-functors}

In \Cref{section:diffeology}, we reviewed three distinct constructions of tangent
spaces on diffeological spaces. All of these constructions extend the classical tangent functor on smooth manifolds, and thus fit into the following general notion.

\begin{defi}[Tangent functor]\label{def:tangentfunctor}
  A \emph{tangent functor} is a functor
  \[
  T \colon \dflgloc \longrightarrow \vect
  \]
  whose restriction to the full subcategory $\enclzeroloc\subset\dflgloc$
  is naturally isomorphic to the classical tangent functor $T^{\enclzeroloc}\colon\enclzeroloc\to\vect$.
\end{defi}  

By \Cref{ex:manifolds}, all three functors in \Cref{section:diffeology} are tangent functors.
Our goal in this section is to construct infinitely many tangent functors
satisfying \Cref{def:tangentfunctor}. 

\subsection{Tangent functors defined by test spaces}

In the Vincent-type construction, the real line $\mathbb{R}$ serves as the basic test space.
Replacing it by an arbitrary based diffeological space $(Y,y)$, we obtain new families of
tangent functors.

\begin{defi}[$(Y,y)$-internal tangent functor]\label{def:Yinternal}
Let $(X,x)$ and $(Y,y)$ be based diffeological spaces.
    The \emph{$(Y,y)$-internal tangent space} of $(X,x)$ is the vector space
    \[
      \generalinternal{(Y,y)}{x}{X}
      = \bigoplus_{p\in\plots(X,x)} T_0(\plotdom{p}) \;\Big/\; R_{(Y,y)},
    \]
    where
    \[
      R_{(Y,y)}
      = \left\langle
        v-w \;\middle|\;
        \begin{aligned}
          &v\in T_0(\plotdom{p}),\; w\in T_0(\plotdom{q}),\\
          &\internalfunctor(f)[p,v] = \internalfunctor(f)[q,w]
            \text{ for all } f\in\Hom_{\dflgloc}((X,x),(Y,y))
        \end{aligned}
      \right\rangle_{\mathrm{vect}}.
    \]
    Here $[p,v]\in \internal{x}{X}$ denotes the class represented by the plot 
    $p:(\plotdom p,0)\to (X,x)$ and the tangent vector $v\in T_0(\plotdom p)$. 
    The map $\internalfunctor(f)\colon \internal{x}{X}\to\internal{y}{Y}$ is the one 
    induced on internal tangent spaces. 
    We denote by $[p,v]_{(Y,y)} \in \generalinternal{(Y,y)}{x}{X}$ the image of 
    $v \in T_0(\plotdom p)$.

For a germ $g\colon (X,x)\to (X',x')$ in $\dflgloc$, define
\[
  \generalinternalfunctor{(Y,y)}(g)
  \colon
  \generalinternal{(Y,y)}{x}{X}
  \longrightarrow
  \generalinternal{(Y,y)}{x'}{X'}\,;\, 
  [p,v]_{(Y,y)}
  \mapsto[g\circ p, v]_{(Y,y)}.
\]
This is well-defined because if $[p,v]$ and $[q,w]$ are identified by $R_{(Y,y)}$, then for every
$h\colon (X',x')\to (Y,y)$ we have
\[
  \internalfunctor(h)[g\circ p,v]
  =
  \internalfunctor(h\circ g)[p,v]
  =
  \internalfunctor(h\circ g)[q,w]
  =
  \internalfunctor(h)[g\circ q,w],
\]
so $[g\circ p,v]$ and $[g\circ q,w]$ are again identified by $R_{(Y,y)}$.
Identities and composition are preserved by construction. Therefore $\generalinternalfunctor{(Y,y)}\colon \dflgloc\to \vect$ is a functor.
\end{defi}

\begin{defi}[$(Y,y)$-right tangent functor]\label{def:Yglobalright}
  Let $(X,x)$ and $(Y,y)$ be based diffeological spaces.
  The \emph{$(Y,y)$-right tangent space} of $(X,x)$ is defined as the linear subspace of $\righttangent{x}{X}$
  \[
    \generalright{(Y,y)}{x}{X}
    = \left\langle
      \rightfunctor(f)D \,\middle|\,
      f\in\Hom_{\dflgloc}((Y,y),(X,x)),\, D\in \righttangent{y}{Y} 
      \right\rangle_{\mathrm{vect}}\subset \righttangent{x}{X}. 
  \]
  Here the map $\rightfunctor(f)\colon \righttangent{y}{Y}\to\righttangent{x}{X}$ is the one 
  induced on right tangent spaces. 

For $g\colon (X,x)\to (X',x')$ in $\dflgloc$, set
\[
  \generalrightfunctor{(Y,y)}(g)
  =
  \rightfunctor(g)\big|_{\generalright{(Y,y)}{x}{X}}
  \,\colon\,
  \generalright{(Y,y)}{x}{X}\ \longrightarrow\ \generalright{(Y,y)}{x'}{X'}.
\]
This is well-defined because for each generator $\rightfunctor(f)D$, we have
\(
  \rightfunctor(g)\big(\rightfunctor(f)D\big)
  =
  \rightfunctor(g\circ f)D,
\)
which again lies in the span defining $\generalright{(Y,y)}{x'}{X'}$.
Since $\rightfunctor$ preserves identities and composition, the restriction inherits these properties. Therefore $\generalrightfunctor{(Y,y)}\colon \dflgloc\to \vect$ is a functor.
\end{defi}

\begin{remark}
When $(Y,y)=(\mathbb{R},0)$, both constructions recover the Vincent-type tangent space
(see \Cref{def:vincent}).
\end{remark}

When $(Y,y)$ has a nontrivial tangent vector, each of the above constructions 
coincides with the classical tangent functor on smooth manifolds.

\begin{prop}\label{prop:Y-internal-and-Y-global-are-tangent-functors}
  Let $(Y,y)$ be a based diffeological space.
  \begin{enumerate}
    \item[(i)] If $\dim \internal{y}{Y}\geq 1$, then $\generalinternalfunctor{(Y,y)}\colon \dflgloc\to\vect$ is a tangent functor (Def.~\ref{def:tangentfunctor}).
    \item[(ii)] If $\dim \righttangent{y}{Y}\geq 1$, then $\generalrightfunctor{(Y,y)}\colon \dflgloc\to\vect$ is a tangent functor (Def.~\ref{def:tangentfunctor}).
  \end{enumerate}
\end{prop}

\begin{proof}
  \textbf{(i)} 
  Let $(M,x)$ be a smooth manifold and suppose that $[p,v]\in\internal{x}{M}$ is nonzero.
  Since the internal tangent space $\internal{x}{M}$ is isomorphic to the standard tangent space, there exists $f\in\germ(M,x)$ with $f(x)=0$ and $d(f\circ p)_0(v)=df_x\bigl([p,v]\bigr)\neq 0$.
  Take a smooth curve $c\colon(\mathbb{R},0)\to(Y,y)$ with 
  $T(c)(\partial/\partial t)\neq0$ (which exists since $\dim\internal{y}{Y}\ge1$),
  and set $F=c\circ f$.
  Then we have
  \[
  T(F)[p,v]=T(c)\!\big(T(f)[p,v]\big)
  =T(c)\!\big(d(f\circ p)_0(v)\big)\neq0,
  \]
  and therefore $[p,v]\ne0$ in $\generalinternal{(Y,y)}{x}{M}$.
  Hence the natural surjection $\internal{x}{M}\to\generalinternal{(Y,y)}{x}{M}$ is also injective,
  and thus $\generalinternalfunctor{(Y, y)}$ is a tangent functor.
  
  \medskip
  \textbf{(ii)}
  Fix a smooth manifold $(M,x)$ and note that $\righttangent{x}{M}\cong T_xM$.
  By definition we have $\generalright{(Y,y)}{x}{M}\subset \righttangent{x}{M}$, so it suffices to prove the reverse inclusion.
  Pick $0\neq D\in\righttangent{y}{Y}$ and choose $s\in\germ(Y,y)$ with $s(y)=0$ and $D(s)=1$ (rescaling if necessary).
  For any $w\in T_xM$, take a smooth curve $\gamma:(\mathbb{R},0)\to(M,x)$ with $\gamma'(0)=w$ and set $f=\gamma\circ s\in\Hom_{\dflgloc}((Y,y),(M,x))$.
  Then for every $g\in\germ(M,x)$, we have
  \[
   (\rightfunctor(f)D)(g)=D\big((g\circ\gamma)\circ s\big).
  \]
  Since $D$ is a derivation on $\germ(Y,y)$, the composition $D(-\circ s)$ is a derivation on $\germ(\mathbb{R},0)$, hence is equal to $c\,\frac{d}{dt}\big|_{t=0}$ with $c=D(-\circ s)(\id_{\mathbb{R}})=D(s)$.
  Thus for any $h\in\germ(\mathbb{R},0)$ we have \(D(h\circ s)=h'(0)\,D(s)\).
  Applying this to $h=g\circ\gamma$ and using $D(s)=1$, we obtain
  \[
   (\rightfunctor(f)D)(g)=(g\circ\gamma)'(0)=dg_x(w).
  \]
  Hence $\rightfunctor(f)D$ is the derivation corresponding to $w$, proving $\righttangent{x}{M}\subset \generalright{(Y,y)}{x}{M}$, and thus $\generalrightfunctor{(Y, y)}$ is a tangent functor. 
\end{proof}

Different choices of the test space $(Y,y)$ in \Cref{def:Yinternal} and \Cref{def:Yglobalright} 
can lead to distinct tangent functors.
In particular, irrational tori already yield uncountably many pairwise non-isomorphic 
$(Y,y)$-internal tangent functors.

\maintheoremone

  \begin{proof}
    Fix $\alpha,\beta\in\mathbb{R}\setminus\mathbb{Q}$ and $x\in T_\alpha$.
    Consider the canonical surjection
    \[
    \rho\colon T_x(T_\alpha)\longrightarrow
    \generalinternal{(T_\beta, y)}{x}{T_\alpha},\quad
    [p,v]\longmapsto [p,v]_{T_\beta}.
    \]
    By \Cref{ex:int-vincent-torus}, the space $T_x(T_\alpha)\cong\mathbb{R}$ is
    one-dimensional and generated by the class $[\pi_\alpha,\partial/\partial t]$.
    Hence $\generalinternal{(T_\beta, y)}{x}{T_\alpha}$ is either $0$ or $\mathbb{R}$,
    depending on whether $\rho([\pi_\alpha,\partial/\partial t])$ vanishes.
    
    If there exists a nonconstant smooth map $f\colon T_\alpha\to T_\beta$,
    then by composing with a translation of $T_\beta$ we may assume $f(x)=y$.
    Then $f$ admits an affine lift $F(t)=at+b$ with $a,b\in\mathbb{R}$ and $a\neq 0$. 
    In this situation we obtain
    \[
      T(f)\bigl([\pi_\alpha,\partial/\partial t]\bigr)
      = [\pi_\beta\circ F,\partial/\partial t]
      = a\,[\pi_\beta,\partial/\partial t]\ \neq\ 0.
    \]
    Conversely, suppose every smooth $f\colon T_\alpha\to T_\beta$ is constant.
    Since the $D$-topology on $T_\alpha$ is indiscrete (\cite[Exercise~55]{IZ}), any smooth germ
    $(T_\alpha,x)\to(T_\beta,y)$ is represented by a global smooth map, and so $T(f)=0$ for all germs $f$.
    Therefore we have $\rho([\pi_\alpha,\partial/\partial t])=0$, which implies $\generalinternal{(T_\beta, y)}{x}{T_\alpha}=0$.    
  \end{proof}
  
    Combining the above theorem with the elementary observation that the equivalence relation generated by
    \[
      \alpha \sim \beta \quad\Longleftrightarrow\quad 
      \alpha = \frac{a+b\,\beta}{c+d\,\beta}\ \text{for some } a,b,c,d\in\mathbb{Z}
    \]
    has uncountably many classes on $\mathbb{R}\setminus\mathbb{Q}$,
    we conclude that there exist uncountably many tangent functors on diffeological spaces,
    pairwise not naturally isomorphic.

    \medskip
    Similarly, an analogous phenomenon occurs for the $(Y,y)$-right tangent functor, where the orbit spaces $\mathbb{R}^n/O(n)$ provide distinct examples.

\maintheoremtwo

\begin{proof}
  For each $k\geq 1$, define
\[
  q_k\colon H_k\longrightarrow \mathbb{R};[x]\mapsto \|x\|^2.
\]
By \Cref{ex:orbit-space}, we have $\righttangent{0}{H_k}\cong\mathbb{R}$, and a generator $D_k$
satisfies $D_k(q_k)\neq 0$.

\medskip
Fix $m\leq n$. Let $\iota:\mathbb{R}^m\hookrightarrow\mathbb{R}^n$ be the isometric embedding $\iota(x)=(x,0)$ and let $\bar\iota:H_m\to H_n$ be the induced smooth map of orbit spaces with $\bar\iota(0)=0$. Since $q_n\circ\bar\iota=q_m$, we have
\[
 (\rightfunctor(\bar\iota)D_m)(q_n)
 = D_m(q_n\circ\bar\iota)
 = D_m(q_m)\neq 0,
\]
so $\rightfunctor(\bar\iota)D_m\neq 0$ in $\righttangent{0}{H_n}$. Therefore, we obtain the chain of inclusions
\[
   \mathbb{R}\cong \langle\,\rightfunctor(\bar\iota)D_m\,\rangle
   \subset \generalright{(H_m,0)}{0}{H_n}
   \subset \righttangent{0}{H_n}\cong \mathbb{R}.
\]

\medskip
Now suppose $m>n$. Let $\pi_k\colon \mathbb{R}^k\to H_k=\mathbb{R}^k/O(k)$ denote the quotient map for each $k$. Let $f\colon (H_m,0)\to (H_n,0)$ be any smooth germ. 
By the definition of the quotient diffeology, there exists an open neighborhood $U\subset\pi_m^{-1}(\dom(f))$ of $0$ and a smooth lift
\[
  F\colon U\to\mathbb{R}^n
  \quad\text{with}\quad
  \pi_n\circ F = f\circ \pi_m \ \text{ on } U.
\]
Then $(q_n\circ f)\circ \pi_m(x)=\|F(x)\|^2$ on $U$.
The function $r\mapsto \|F(r,0,\dots,0)\|^2$ on $\{r\in\mathbb{R}\mid (r,0,\dots,0)\in U\,\}$ is smooth and even,
so there exists a smooth $\Psi\colon\mathbb{R}\to\mathbb{R}$ such that
$\|F(r,0,\dots,0)\|^2=\Psi(r^2)$.
Since $\|F(x)\|^2$ is $O(m)$-invariant, we have
\[
  \|F(x)\|^2=\|F(\|x\|,0,\dots,0)\|^2=\Psi(\|x\|^2)\qquad (x\in U).
\]
We write the Taylor expansions at $0$ as
\[
  F(x)=Ax+o(\|x\|)\quad (x\to 0),\qquad \Psi(t)=at+o(t)\quad (t\to 0).
\]
Comparing quadratic terms gives, for all $x\in\mathbb{R}^m$,
\[
  {}^t\!x\,({}^t\!A A)\,x = a\|x\|^2,
  \quad\text{i.e.}\quad {}^t\!A A = aI_m.
\]
Since $\operatorname{rank}({}^t\!A A)\leq \operatorname{rank}(A)\le n<m$, the matrix
${}^t\!A A$ cannot be a nonzero scalar multiple of $I_m$.
Hence $a=0$ and $A=0$.
Therefore we have
\[
  \frac{\partial^2}{\partial x_1^2}\,\|F(x)\|^2\Big|_{x=0}
  = \frac{\partial^2}{\partial x_1^2}\,\left(a\|x\|^2+o(\|x\|^2)\right)\Big|_{x=0} = 2a = 0,
\]
so we obtain
\[
  (\rightfunctor(f)D_m)(q_n)
  = D_m(q_n\circ f)
  = \frac12\,\frac{\partial^2}{\partial x_1^2}\big((q_n\circ f)\circ \pi_m\big)(0)
  = 0.
\]
Since $f$ is arbitrary, $\generalright{(H_m,0)}{0}{H_n}=0$ for $m>n$.
\end{proof}

Hence even within the right-type construction, infinitely many distinct tangent functors appear.
Informally, $\generalrightfunctor{(H_m,0)}$ measures the “sharpness” of singularities:
it becomes trivial when the source space is sharper than the target.

\section*{Acknowledgments}
    I would like to express my sincere gratitude to my supervisor, Takuya Sakasai, for his valuable guidance and continuous support throughout this research. 
    I am also grateful to Katsuhiko Kuribayashi for many insightful discussions and helpful advice. 
    I thank Toshiyuki Kobayashi for his generous support and encouragement as my supporting supervisor in the WINGS-FMSP program.  
    This work was supported by JSPS Research Fellowships for Young Scientists and KAKENHI Grant Number JP24KJ0881. 
    Lastly, I would like to acknowledge the WINGS-FMSP program for its financial support.

\bibliographystyle{plain}
\bibliography{infinitetangent}



\end{document}